\documentclass[openany,reqno,a4paper,12pt]{amsart}
\usepackage{amsmath,amssymb,amsthm}

\usepackage{mathrsfs}

\usepackage{mathtools}
\usepackage{commath}

\usepackage{thmtools}
\usepackage{thm-restate}

\usepackage{cases}
\usepackage{enumitem}
\setlist[enumerate,1]{label=(\roman*)}

\usepackage[pdftex, pdfborderstyle={/S/U/W 0}]{hyperref}
\hypersetup{
    colorlinks=true,
    linkcolor=magenta,
    citecolor=cyan,
}

\usepackage{cleveref}

\usepackage{etoolbox}

\numberwithin{equation}{section}

\ifdefined\thmcolor
\declaretheoremstyle[
  shaded={bgcolor=\thmcolor}
]{plain}
\else
\fi

\ifdefined\defcolor
\declaretheoremstyle[
  headfont=\normalfont\bfseries,
  bodyfont=\normalfont,
  shaded={bgcolor=\defcolor}
]{noital}
\else
\declaretheoremstyle[
  headfont=\normalfont\bfseries,
  bodyfont=\normalfont,
]{noital}
\fi

\declaretheorem[style=plain,numberwithin=section,name=Theorem]{theorem}

\declaretheorem[style=plain,sibling=theorem,name=Lemma]{lemma}

\declaretheorem[style=plain,sibling=theorem,name=Conjecture]{conjecture}

\declaretheorem[style=plain,sibling=theorem,name=Question]{question}
\declaretheorem[style=plain,sibling=theorem,name=Observation]{observation}

\declaretheorem[style=plain,numbered=no,name=Theorem]{theorem-n}
\declaretheorem[style=plain,numbered=no,name=Proposition]{proposition-n}
\declaretheorem[style=plain,numbered=no,name=Lemma]{lemma-n}
\declaretheorem[style=plain,numbered=no,name=Corollary]{corollary-n}
\declaretheorem[style=plain,numbered=no,name=Conjecture]{conjecture-n}
\declaretheorem[style=plain,numbered=no,name=Claim]{claim-n}
\declaretheorem[style=plain,numbered=no,name=Fact]{fact-n}
\declaretheorem[style=plain,numbered=no,name=Open Problem]{openproblem-n}
\declaretheorem[style=plain,numbered=no,name=Question]{question-n}

\declaretheorem[style=noital,sibling=theorem,name=Definition]{definition}

\declaretheorem[style=noital,numbered=no,name=Remark]{remark-n}
\declaretheorem[style=noital,numbered=no,name=Definition]{definition-n}
\declaretheorem[style=noital,numbered=no,name=Construction]{construction-n}
\declaretheorem[style=noital,numbered=no,name=Example]{example-n}

\newcommand{\defined}{\mathrel{\coloneqq}}

\newcommand{\st}{\mathbin{\colon}}

\undef{\set}
\DeclarePairedDelimiter{\set}{\lbrace}{\rbrace}

\undef{\emptyset}
\newcommand{\emptyset}{\varnothing}

\newcommand{\union}{\mathbin{\cup}}
\newcommand{\inter}{\mathbin{\cap}}

\newcommand{\from}{\colon}

\DeclarePairedDelimiter{\floor}{\lfloor}{\rfloor}
\DeclarePairedDelimiter{\ceil}{\lceil}{\rceil}

\undef{\abs}
\DeclarePairedDelimiterX{\abs}[1]
  {\lvert}{\rvert}{\ifblank{#1}{\,\cdot\,}{#1}}

\undef{\norm}
\DeclarePairedDelimiterX{\norm}[1]
  {\lVert}{\rVert}{\ifblank{#1}{\,\cdot\,}{#1}}

\DeclarePairedDelimiterX{\inner}[2]
  {\langle}{\rangle}{\ifblank{#1}{\,\cdot\,}{#1},\ifblank{#2}{\,\cdot\,}{#2}}

\DeclareMathDelimiter{\given}
  {\mathbin}{symbols}{"6A}{largesymbols}{"0C}

\DeclareMathOperator{\Prob}{\mathbb{P}}
\DeclarePairedDelimiterXPP{\prob}[1]
  {\Prob}{\lparen}{\rparen}{}
  {\renewcommand{\given}{\nonscript\;\delimsize\vert\nonscript\;\mathopen{}}#1}

\DeclareMathOperator{\Expec}{\mathbb{E}}
\DeclarePairedDelimiterXPP{\expec}[1]
  {\Expec}{\lparen}{\rparen}{}
  {\renewcommand{\given}{\nonscript\;\delimsize\vert\nonscript\;\mathopen{}}#1}

\DeclareMathOperator{\Var}{Var}
\DeclarePairedDelimiterXPP{\var}[1]
  {\Var}{\lparen}{\rparen}{}
  {\renewcommand{\given}{\nonscript\;\delimsize\vert\nonscript\;\mathopen{}}#1}

\DeclareMathOperator{\Cov}{Cov}
\DeclarePairedDelimiterXPP{\cov}[2]
  {\Cov}{\lparen}{\rparen}{}{#1,#2}

\newcommand{\eps}{\varepsilon}
\newcommand{\sseq}{\subseteq}

\newcommand{\NN}{\mathbb{N}}

\newcommand{\cC}{\mathcal{C}}
\newcommand{\cD}{\mathcal{D}}

\usepackage[numbers]{natbib}

\usepackage{geometry}
\DeclareMathOperator{\col}{col}
\newcommand{\chrom}{\chi_\text{g}}
\newcommand{\blanks}{\chi_\text{gb}}

\newcommand{\ts}{\textsuperscript}

\begin{document}
\title[On graphs with maximum difference between $\chrom(G)$ and $\chi(G)$]{On graphs with maximum difference between game chromatic number and chromatic number}

\author{Lawrence Hollom}
\address{Department of Pure Mathematics and Mathematical Statistics (DPMMS), University of Cambridge, Wilberforce Road, Cambridge, CB3 0WA, United Kingdom}
\email{lh569@cam.ac.uk}

%\keywords{Combinatorics, Mathematics.}

%\thanks{Put current funding information here.}

\begin{abstract}
    In the vertex colouring game on a graph $G$, Maker and Breaker alternately colour vertices of $G$ from a palette of $k$ colours, with no two adjacent vertices allowed the same colour.
    Maker seeks to colour the whole graph while Breaker seeks to make some vertex impossible to colour.
    The game chromatic number of $G$, $\chrom(G)$, is the minimum number $k$ of colours for which Maker has a winning strategy for the vertex colouring game.
    
    Matsumoto proved in 2019 that $\chrom(G)-\chi(G)\leq\floor{n/2} - 1$, and conjectured that the only equality cases are some graphs of small order and the Tur\'{a}n graph $T(2r,r)$.
    We resolve this conjecture in the affirmative by considering a modification of the vertex colouring game wherein Breaker may remove a vertex instead of colouring it.

    Matsumoto further asked whether a similar result could be proved for the vertex marking game, and we provide an example to show that no such nontrivial result can exist.
\end{abstract}

\maketitle

%----------------------------------------------------------%

\section{Introduction}
\label{sec:intro}

The vertex colouring game on a graph is a widely studied combinatorial game, first invented by Brams and then discussed by Gardner in his ``Mathematical Games'' column of Scientific American \cite{gardner1981scientific} in 1981. 
The game was subsequently reinvented by Bodlaeder \cite{bodlaender1991complexity} in 1991, who brought it into the combinatorics literature, in which it has since attracted significant attention.

The game is played on a graph $G$ by two players, Maker and Breaker, often also referred to as Alice and Bob.
The players alternately assign colours from a set $X$ of size $k$ to vertices of the graph, Maker going first, with the restriction that no two adjacent vertices may be given the same colour.
Maker wins the game if the whole of $G$ can be coloured, and Breaker wins otherwise, i.e. if some vertex becomes impossible to colour.
The \emph{game chromatic number} $\chrom(G)$ of $G$ is then the minimal number $k$ of colours for which Maker has a winning strategy.

Calculating -- or even estimating -- $\chrom(G)$ is very difficult, even for small graphs.
Bodlaeder first asked about the computational complexity of the vertex colouring game in 1991, and it took until 2020 before Costa, Pessoa, Sampaio, and Soares \cite{costa2020pspace} proved that the both the vertex colouring game and the related greedy colouring game are PSPACE-complete.
This result has since been extended to variants of the game where Breaker starts, or one or other of the players can pass their turn \cite{marcilon2020hardness}.

Despite the difficulty of the problem in general, the game chromatic number of various classes of graphs has received significant attention; the cases of forests \cite{faigle1991game} and planar graphs \cite{kierstead1994planar, nakprasit2018game, sekiguchi2014game, zhu1999game, zhu2008refined} have been studied, as have cactuses \cite{sidorowicz2007game}, and the relation to the acyclic chromatic number \cite{dinski1999bound} and game Grundy number \cite{havet2013game}, to name a few avenues of research.
Another significant direction of study, initiated by Bohman, Frieze, and Sudakov \cite{bohman2008game}, is the study of the game chromatic number of the binomial random graph.

One approach that can be taken in bounding $\chrom(G)$ is to compare it to the usual chromatic number $\chi(G)$.
Such an approach was taken by Matsumoto \cite{matsumoto2019difference}, who in 2019 proved the following theorem.

\begin{theorem}[{\cite[Theorem 1]{matsumoto2019difference}}]
\label{thm:comparison-basic}
    For any graph $G$ with $n\geq 2$ vertices, we have 
    \begin{align*}
        \chrom(G)-\chi(G)\leq\floor{n/2}-1.
    \end{align*}
\end{theorem}

Following on from this, Matsumoto conjectured a classification of the equality cases, as follows.

\begin{conjecture}[{\cite[Conjecture 2]{matsumoto2019difference}}]
\label{conj:matsumoto}
    For any graph $G$ with $n\geq 4$ vertices,
    \begin{align*}
        \chrom(G)-\chi(G)\leq\floor{n/2}-2,
    \end{align*}
    unless $G$ is isomorphic to either the Tur\'{a}n graph $T(2r,r)$ or $K_{2,3}$.
\end{conjecture}

Throughout this paper, all graphs we consider are finite and simple, but not necessarily connected.
We resolve \Cref{conj:matsumoto} (modulo some trivialities regarding graphs of small order) by first introducing a new, modified version of the vertex colouring game, wherein Breaker may choose to play a `blank', which is equivalent to deleting a vertex.
Neighbours of a vertex at which a blank has been played may still be given any colour, or even another blank.
This game, and the corresponding graph invariant $\blanks(G)$ (the minimal number of colours, not counting blanks, for which Maker wins) are significantly more amenable to inductive methods than the standard vertex colouring game.
It is this property which allows us to prove the following theorem.

\begin{theorem}
\label{thm:intro-comparison}
    For any graph $G$ with $n\geq 6$ vertices which is not isomorphic to the $r$-partite Tur\'{a}n graph $T(2r,r)$, we have 
    \begin{align}
        \blanks(G)-\chi(G)\leq\floor{n/2}-2.
    \end{align}
\end{theorem}

Noting that $\blanks(G)\geq \chrom(G)$, \Cref{conj:matsumoto} follows from \Cref{thm:intro-comparison}.

Matsumoto also asked whether a similar inequality holds for the marking game, also called the graph colouring game (see \Cref{sec:marking-game} for a definition of this game).
We resolve this problem, presenting graphs with large game colouring number and small chromatic number in \Cref{sec:marking-game}.

In \Cref{sec:preliminaries} we give the definitions and preliminary results we shall need in our proof of \Cref{thm:intro-comparison}, before presenting the inductive step of the proof in \Cref{sec:main-proof}.
We conclude with some open problems and directions for future work in \Cref{sec:future}.

%----------------------------------------------------------%

\section{Definitions and preliminary results}
\label{sec:preliminaries}

The precise small graphs for which \Cref{conj:matsumoto} does not hold are the $r$-partite Tur\'{a}n graph $T(2r,r)$, any of the connected subgraphs of the complete bipartite graph $K_{2,3}$ which itself contains the path on four vertices $P_4$ as a subgraph, and any graph of order at most 3.
We restrict our attention to graphs of order at least 6 to avoid the technicalities that come with these small cases.
We in fact prove a statement slightly stronger than \Cref{thm:intro-comparison} concerning a generalisation of the vertex colouring game, which we now define.

\begin{definition}
\label{def:game-with-blanks}
    We define the \emph{vertex colouring game with blanks} to be the usual vertex colouring game, except now Breaker can use his turn to mark a vertex as `blank', rather than with a particular colour.
    When we refer to `colours' in this game, we do not include blanks.
    We will refer to a player marking a vertex as blank as \emph{playing a blank} at that vertex.
    Blanks do not restrict what colours may be played at adjacent vertices, and any such vertex may still be given any colour not prohibited by its own neighbours.
    Thus playing a blank is equivalent to deleting a vertex from the graph.
    The game ends when no vertex can legally be given any (non-blank) colour; in particular, if Breaker's only legal move is to play a blank, then the game ends and Breaker wins.
    The minimum number of colours for which Maker wins the vertex colouring game with blanks is denoted $\blanks(G)$.
\end{definition}

Throughout the rest of the paper we will refer to a player assigning a colour to a vertex as \emph{playing a vertex}.
We now generalise the game even further, as is necessary for our proof.

\begin{definition}
\label{def:marked-for-blanks}
    The vertex colouring game with blanks on $G$ with independent sets $D_1,\dotsc,D_s \sseq V(G)$ \emph{marked for blanks} is played as the vertex colouring game with blanks, but with the following differences.
    Now Maker -- as well as Breaker -- may play blanks in $D\defined D_1\union\dots\union D_s$, and the game does not end while vertices in $D$ are left unplayed, even if the only valid move is to play a blank.
    Note that neither player is forced to play in $D$. 
    Furthermore, Breaker may still play blanks at any vertex, whether it is in $D$ or not, and both Maker and Breaker can play both colours and blanks in $D$.
    Finally, whenever Breaker plays a blank at a vertex $v\in V(G)$ (and only if he plays a blank), he may pick some class marked for blanks $D_i\not\ni v$ (if such a class exists), and remove it from the list of classes marked for blanks.
    
    We define $\blanks(G;D_1,\dotsc,D_s)$ to be the minimum number of colours for Maker to win the vertex colouring game with blanks, with independent sets $D_1,\dotsc,D_s$ marked for blanks.
    
    In particular, if $s=0$ and no classes are marked for blanks then the above game agrees with \Cref{def:game-with-blanks}. 
    At the other extreme, if $D_1=V(G)$ (and so $G$ has no edges), $\blanks(G;D_1)=0$, as both players are forced to play blanks until there are no vertices left.
    
    Finally, we will on occasion want to condition on the first few plays of the game.
    If $P$ is an ordered list of (vertex, colour-or-blank) pairs which form a legal sequence of moves, then
    $\blanks(G;D_1,\dotsc,D_s|P)$ is given the same definition as $\blanks(G;D_1,\dotsc,D_s)$, but with the added condition that the game must start with the sequence $P$ of moves.\\
\end{definition}

Intuitively, the purpose of classes marked for blanks is to allow us to induct: if some colour $c$ has been played only in some independent set $D$, then it should remain possible to play $c$ in $D$.
This is represented by removing played vertices from $G$ and removing used colours from the palette, and marking $D$ for blanks.
After played vertices and colours (including $c$) have been removed, either Maker or Breaker playing a blank in $D$ can be thought of as playing colour $c$.
This intuition is formalised in \Cref{lem:subgraph-imagination}.

Note that classes marked for blanks only help Maker: Maker can choose to play as if there are no classes marked for blanks, and the game will proceed exactly as if there are no such classes.
The only exception to this is at the end of the game, when both players may be forced to play blanks, allowing Maker a chance to win when she would otherwise have lost.
Therefore $\blanks(G)\geq\blanks(G;D_1,\dotsc,D_s)$.

We now state our main theorem, which the remainder of this section and \Cref{sec:main-proof} are devoted to proving.

\begin{theorem}
\label{thm:comparison-better}
    Let $G$ be a graph on $n$ vertices with $n\geq 2$, and $C_1,\dotsc,C_p$ be a partition of $V(G)$ into $p\geq\chi(G)$ disjoint independent sets. 
    Let $0\leq s\leq p$, and let $D_1,\dotsc,D_s$ be some $s$ of the sets $C_1,\dotsc,C_p$.
    Assume that either $s\geq 1$, or that $n\geq 6$ and $G$ is not isomorphic to the Tur\'{a}n graph $T(2r,r)$.
    Then
    \begin{align}
    \label{eq:comparison-target}
        \blanks(G;D_1,\dotsc,D_s)\leq p+\floor{n/2}-2.
    \end{align}
\end{theorem}

Noting that we may take $s=0$ and $p=\chi(G)$ in the above, \Cref{thm:intro-comparison} follows.
We shall refer to the independent sets $C_1.\dotsc,C_p$ as the `classes' of $G$, noting that they are colour classes in some proper colouring of $G$ (not necessarily using the minimal number of colours).
Our proof of \Cref{thm:comparison-better} proceeds by induction on $n$, and we use the following as our base case.

\begin{lemma}
\label{lem:base-case}
    All of the following finite simple graphs satisfy inequality \eqref{eq:comparison-target} with $s=0$.
    \begin{enumerate}
        \item \label{case:matsumoto-6} Graphs of order 6 except for $T(6,3)$.
        \item \label{case:matsumoto-7} Graphs of order 7.
        \item \label{case:matsumoto-interesting} Proper subgraphs $G\subsetneqq T(2r,r)$ where $r\geq 3$, $\abs{V(G)}=2r$, and each $C_i$ has size 2.
        \item \label{case:matsumoto-bipartite} Any subgraph of $K_{2,r}$ on $2+r$ vertices, with $p=2$ and $\abs{C_1}=2$, for $r\geq 4$.
    \end{enumerate}
    Furthermore, the following satisfy \eqref{eq:comparison-target} with $s\geq 1$.
    \begin{enumerate}[resume]
        \item \label{case:matsumoto-2} Graphs of order 2.
        \item \label{case:matsumoto-3} Graphs of order 3.
    \end{enumerate}
    Finally, the following graphs satisfy $\blanks(G)\leq \chi(G) + \floor{n/2} - 1$.
    \begin{enumerate}[resume]
        \item \label{case:matsumoto-base} Graphs of order $n$ for $2\leq n \leq 5$.
    \end{enumerate}
\end{lemma}

The proofs of parts \ref{case:matsumoto-6}, \ref{case:matsumoto-7}, \ref{case:matsumoto-2}, \ref{case:matsumoto-3}, and \ref{case:matsumoto-base} of \Cref{lem:base-case} are technical and not enlightening, and so are given in \Cref{sec:appendix}.
Here we present first the proof of part \ref{case:matsumoto-interesting} of \Cref{lem:base-case}, and then deal with part \ref{case:matsumoto-bipartite}.

\begin{proof}[Proof of \Cref{lem:base-case} part \ref{case:matsumoto-interesting}]
    Let $G\subsetneqq T(2r,r)$ satisfy $\abs{V(G)}=2r$ and $\chi(G) \leq r$, and let $C_1=\set{x_1,y_1},\dotsc,C_r=\set{x_r,y_r}$ be the colour classes in an arbitrary proper $r$-colouring of $G$.
    We present a strategy for Maker to win the vertex colouring game with $2r-2$ colours.

    Note first that for Maker to win, there only need to be two times throughout the game when a colour is repeated or a blank is played.
    Consider the final play of the game, which is made by Breaker; say this play is at vertex $y_r$.
    The only reason why $y_r$ could not be given the same colour as $x_r$ is if there was some other vertex which already had the same colour as $x_r$.
    Thus it suffices for Maker to force a single colour to be repeated, with both instances of that colour occurring in classes which are completely coloured before the final play of the game.
    For the same reason, if Breaker ever plays a blank before Maker's final play of the game, then Maker will win, so we may assume that Breaker does not play blanks.
    
    Since $G$ is a proper subgraph of $T(2r,r)$, there must be some edge missing between colour classes.
    Relabelling if necessary, we may assume that edge $x_1x_2$ is absent.
    Maker's strategy is to give vertex $x_1$ colour 1, and we split into cases based on Breaker's response.
    
    If Breaker replies by playing at vertex $y_1$, then either he plays colour 1, in which case a colour is repeated and we are done, or he plays a different colour, in which case Maker plays colour 1 at vertex $x_2$.
    Wherever Breaker plays next, Maker then uses her next turn to ensure vertex $y_2$ is coloured with any colour, which is sufficient for Maker to win.

    If Breaker plays colour 1 at some vertex other than $y_1$, which we may w.l.o.g. assume is $x_2$, then Maker uses her next two turns to ensure vertices $y_1$ and $y_2$ are coloured (with arbitrary colours), which is again sufficient to win.

    If Breaker plays some other colour, say colour 2, at some vertex $x_i$ for $i\geq 2$, then Maker can reply by giving $y_i$ colour 2, which is again sufficient to win.

    Thus when there are $2r-2$ colours available, Maker has a winning strategy, as required.
\end{proof}

Next we deal with \Cref{lem:base-case} part \ref{case:matsumoto-bipartite}.
This is achieved via the following result, which establishes a greedy strategy for Maker.

\begin{lemma}
\label{lem:greedy}
    If a graph $G$ has at most $\ceil{k/2}$ vertices of degree at least $k$, then $\blanks(G)\leq k$.
\end{lemma}

\begin{proof}
    Assume that $k$ colours are available.
    Maker's strategy on her first $\ceil{k/2}$ turns is to arbitrarily colour any vertex of degree at least $k$.
    Note that if all such vertices are coloured, then Maker wins, as a vertex of degree strictly less than $k$ will always be playable.

    Furthermore, just before Maker's $i$\ts{th} turn, $2(i-1)$ vertices have been coloured.
    Thus on Maker's $\ceil{k/2}$\ts{th} turn, $2(\ceil{k/2}-1)\leq k-1$ vertices have been coloured, so all vertices are still colourable.
    Therefore all vertices of degree at least $k$ can be coloured, and so this is a winning strategy for Maker.
\end{proof}

Note that applying \Cref{lem:greedy} with $k=3$ immediately implies that for any subgraph $H$ of $K_{2,r}$, $\blanks(H)\leq 3$, and \Cref{lem:base-case} part \ref{case:matsumoto-bipartite} follows.

We now state and prove two further lemmas which we will need in our proof.
The first lemma is the tool which allows us to use induction, which will be needed for our proof of \Cref{thm:comparison-better}.
Our proof of this lemma makes use of a so-called \emph{imagination strategy}, as popularised by the seminal paper of Bre{\v{s}}ar, Klav{\v{z}}ar, and Rall \cite{brevsar2010domination}.

\begin{lemma}
\label{lem:subgraph-imagination}
    Assume that Maker and Breaker have each played $t$ times in the vertex colouring game with blanks, forming some sequence $P$ of plays, and that $U\sseq V(G)$ is the set of vertices which have been played (so $\abs{U}=2t$). 
    Let $X$ be the set of colours that have been played in $U$ (not including blanks) and let $c_1,\dotsc,c_s\in X$ be some distinct colours.
    Finally, let $D_1,\dotsc,D_s,E_1,\dots,E_q\sseq V(G)$ be some disjoint non-empty independent sets, where $D_i$ includes all vertices which have been given colour $c_i$ for each $1\leq i \leq s$. 
    Then if we let $G'=G-U$ be the subgraph induced by the uncoloured vertices, then we have 
    \begin{align}
    \label{eq:induction}
        \blanks(G;E_1,\dotsc,E_q|P)\leq \blanks(G';D_1',\dotsc,D_s',E_1',\dotsc,E_q')+\abs{X}.
    \end{align}
    Where $D_i'\defined D_i\inter V(G')$ and $E_j'$ is either $E_j\inter V(G')$, or $\emptyset$ if Breaker removed $E_j$ as a class marked for blanks, for all $1\leq i\leq s$ and $1\leq j\leq q$.
\end{lemma}

\begin{proof}
    We prove this via an imagination strategy, constructing a winning strategy for Maker on $G$ with $E_1,\dotsc,E_q$ marked for blanks, starting as in $P$, using $\blanks(G';D_1',\dotsc,D_s',E_1',\dotsc,E_q')+\abs{X}$ colours.
    After the first $2t$ plays of the game in $G$, Maker imagines playing on $G'\sseq G$ with sets $D_1',\dotsc,D_s',E_1',\dotsc,E_q'$ marked for blanks using the remaining colours.
    Call the set of remaining colours $Y$ (so $X$ and $Y$ are disjoint).

    In short, Maker playing a blank in $D_i'$ in the imagined game is equivalent to her playing colour $c_i$ in the real game.

    To move Maker's plays from the imagined game on $G'$ to the real game on $G$ we do the following.
    If Maker plays a colour from $Y$ or a blank in some $E_j'$, then this move may be copied to $G$.
    Otherwise, Maker played a blank at $v\in D_i'$. 
    In the real game on $G$, we instead let Maker play colour $c_i$ at $v$, which is playable on vertices in $D_i$ so long as Breaker has not played colour $c_i$ from move $2t+1$ onward.
    
    To move Breaker's plays from the real game on $G$ to the imagined game on $G'$, we do the following.
    
    If Breaker plays a colour from $Y$ or a blank, then this move is copied over into the imagined game.
    If Breaker plays a blank and removes class $E_i$ from the list of classes marked for blanks, then $E_i'$ is also removed in the imagined game.

    If Breaker plays a colour from $X$, we instead have Breaker play a blank in the imagined game.
    In particular, if Breaker plays $c_i\in X$ at some vertex $v$ outside of $D_i$, then in the imagined game Breaker plays a blank at $v$ and remove $D_i'$ from the list of classes marked for blanks.
    Otherwise Breaker does not remove a class marked for blanks.
    
    This means that if Breaker plays colour $c_i$ outside of $D_i$ from move $2t+1$ onward, possibly preventing colour $c_i$ from being played in $D_i'$, then in the imagined game Breaker will play a blank outside of $D_i'$ and disallow Maker from playing blanks in $D_i'$.
    
    Thus if Maker plays a blank at $x\in D_i'\sseq G'$, then we know that Breaker has not yet played colour $c_i$ from move $2t+1$ onward in the real game, so playing colour $c_i$ at $x$ in the real game is a valid move.

    Therefore Maker's strategy can be converted from the imagined game to the real game, and the result follows.
\end{proof}

Our second lemma deals with a particular case in the induction, wherein the subgraph induced by the uncoloured vertices is isomorphic to $T(2r,r)$.

\begin{lemma}
\label{lem:annotated-turan}
    If $G$ is a $2r$-vertex subgraph of $T(2r,r)$ for some $r$, and $D_1,\dotsc,D_s$ are some $s\geq 0$ of the vertex classes corresponding to colour classes in the unique (up to relabelling colours) $r$-colouring of $T(2r,r)$, then
    \begin{align}
        \blanks(G;D_1,\dotsc,D_s)\leq 2r - s - 1.
    \end{align}
\end{lemma}

\begin{proof}
    We present a strategy for Maker whereby, for each $1\leq i\leq s$, either Maker plays a blank in $D_i$ or Breaker plays a blank to remove $D_i$ as a class marked for blanks, and from this Maker wins with $2r-s-1$ colours.

    Maker's strategy while there is still a class marked for blanks is to always play a blank (if possible), playing in the same class as Breaker's previous move if that class is not already fully coloured, and in an arbitrary class marked for blanks otherwise.
    Maker's aim is for each original class $D_i$ marked for blanks to have some blank played either in $D_i$ by Maker, or outside of $D_i$ by Breaker, who then removed $D_i$ from the list of classes marked for blanks.
    
    Consider some such $D_i$.
    The only way Maker could fail to achieve her goal is if $D_i$ was filled before she had a chance to play in it, and Breaker never removed it as a class marked for blanks.
    But as soon as Breaker plays in $D_i$, Maker will play a blank in $D_i$, and so Maker succeeds in her goal.

    In particular, at least $s$ blanks are played in the game, so it suffices to prove that either $s+1$ blanks were played or some colour was repeated.

    Consider the final move of the game, which is made by Breaker since there are an even number of vertices.
    Assume that all colours have been played at least once, as otherwise Breaker has a legal move, and Maker wins.
    Say it is at a vertex $x$ in some class $C = \set{x,y}$, which may or may not be some $D_i$.
    The vertex $y$ has already been played. 
    If $y$ was given some colour $c$ already, then either Breaker can give $x$ colour $c$, and Maker wins, or colour $c$ has already been given to another vertex in $G$, in which case Maker also wins.
    If $y$ has been played and is blank, then note that this blank must have been played by Maker, and $C=D_i$ for some $i$, as otherwise Maker would already have played at $x$.
    But then either playing a blank is a legal move at $y$, meaning at least $s+1$ blanks are played, or Breaker played a blank which removed $C$ as a class marked for blanks.
    But in this case there were two blanks associated with class $C$, and so there must be a colour remaining for Breaker to play.

    Thus Maker wins with $2r-s-1$ colours, as required.
\end{proof}

Together with the proofs in \Cref{sec:appendix} we now have all the preliminary results we require.
As \Cref{lem:base-case} deals with the base case of the proof of \Cref{thm:comparison-better}, only the inductive step remains, which we present in the following section.

\section{Proof of Theorem \ref{thm:comparison-better}}
\label{sec:main-proof}

Let $G=(V,E)$ and let $V=C_1\union\dots\union C_p$ be a partition of $V$ into $p$ independent sets.
We shall proceed by producing the first few moves of Maker's strategy, and then remove played vertices to proceed by means of \Cref{lem:subgraph-imagination} and \Cref{lem:annotated-turan}, as required for the induction.
When we apply \Cref{lem:subgraph-imagination}, we will do so in such a way that a class marked for blanks is one of the independent sets $C_1,\dots,C_p$.
The base case of our induction is provided by \Cref{lem:base-case}, and so here we present only the inductive step.

We first present in \Cref{subsec:maker-strat} an outline of our proof, the notation we will use, and some of the strategies Maker will employ.
We then demonstrate that the strategy works in the following two sections, conditioning on Maker's first move.
The case in which Maker plays a blank is dealt with in \Cref{subsec:maker-plays-blank}, and then the case in which Maker plays a colour is dealt with in \Cref{subsec:maker-plays-colour}.

\subsection{Notation and setup}
\label{subsec:maker-strat}

If there are $s\geq 1$ nonempty sets $D_1,\dotsc,D_s$ marked for blanks as in the statement of \Cref{thm:comparison-better}, with $\abs{D_1}\leq\dots\leq\abs{D_s}$, then Maker plays a blank in $D_1$ as the first move.

Assuming instead that $s=0$, if there is some class with odd size, then Maker plays colour 1 at a vertex $u\in C_i$ of maximal degree in its class, where $C_i$ is a class of minimal odd size.

Otherwise, either \Cref{lem:annotated-turan} applies, or there is a class of size at least 4, and Maker plays arbitrarily in a class of minimum even size.

In any case, we may w.l.o.g. say that Maker played at vertex $u$ in class $C_1$, and that Breaker replies at a vertex $v$ in either $C_1$ or $C_2$.

If further moves of Maker's strategy need to be specified, they depend on Breaker's reply to Maker's first move, and so are detailed in the analysis to follow.
One strategy which may be employed by Maker in her second turn and onward is to play in the same class as Breaker's most recent move, using the same colour.
We refer to such a move as Maker `copying Breaker'.

Our proof will proceed by conditioning on the first $2t$ moves of the game (i.e. $t$ moves for each of Maker and Breaker), and then applying \Cref{lem:subgraph-imagination} and our inductive hypothesis.
Let $P$ be a list of these plays, and denote by $G'$ the graph resulting from our application of \Cref{lem:subgraph-imagination}.
Furthermore, let $X$ be the set of colours played on $V(G)\setminus V(G')$, and $n$ and $n'$ be the orders of $G$ and $G'$ respectively, with $n-n' = 2t$.
Let $\cD$ and $\cD'$ be the sets of nonempty classes marked for blanks in $G$ and $G'$ respectively.
Finally, let $p'=\abs{\set{i\st C_i\inter G'\neq \emptyset}}$ be the number of independent sets in our partition of $G'$.

Our task is to construct a strategy for Maker so that, in the cases for which we cannot conclude our result directly, $G'$ satisfies the assumptions of \Cref{thm:comparison-better}, and so we can conclude by induction.
We now spell out this method in more detail.

We will find some $t\geq 1$ and a strategy for Maker's first $t$ moves so that either we can directly conclude that inequality \eqref{eq:comparison-target} holds for $G$, or that $n'\geq 2$, and the following two statements hold.
\begin{align}
\label{eq:strange-induction-condition}
    (G'\not\simeq T(2r,r)\text{ and }n'\geq 6)\text{ or }\abs{\cD'}\geq 1\text{ or }p-p'+(n-n')/2 \geq \abs{X} + 1.
\end{align}
\begin{align}
\label{eq:general-induction-target}
    (p-p')+(n-n') / 2 \geq \abs{X} + \max(\abs{\cD} - \abs{\cD'},0).
\end{align}

We now show why these suffice to conclude \Cref{thm:comparison-better}.

Firstly, as $n'\geq 2$, $G'$ satisfies either case \ref{case:matsumoto-base} of \Cref{lem:base-case} (if $n'\leq 5$), or the conclusion of \Cref{thm:comparison-better} (if $n'\geq 6$, by induction), so, noting that $\chi(G')\leq p'$ and combining with \Cref{lem:subgraph-imagination,lem:annotated-turan}, we find that
$$\blanks(G;D_1,\dotsc,D_s|P)\leq \abs{X} + p' + \floor{n'/2} - 1.$$

If $\abs{\cD'} < \abs{\cD}$, then we know that $\max(\abs{\cD} - \abs{\cD'}, 0)\geq 1$.
Also, recalling that $n-n'$ is even, inequality \eqref{eq:general-induction-target} implies that
$$0\leq p-p' + \floor{n/2} - \floor{n'/2} -\abs{X} - 1,$$
and so we deduce that
\begin{align}
\label{eq:conditioned-induction-target}
    \blanks(G;D_1,\dotsc,D_s|P)\leq p+\floor{n/2}-2.
\end{align}

If $\abs{\cD'} \geq \abs{\cD}$, then condition \eqref{eq:strange-induction-condition} implies that $G'$ either satisfies the assumptions of \Cref{thm:comparison-better} (if either of the first two clauses hold), or inequality \eqref{eq:conditioned-induction-target} holds (if the third clause holds).
In the case that $G'$ satisfies the assumptions of \Cref{thm:comparison-better}, we may apply the induction hypothesis and \Cref{lem:subgraph-imagination} to find the following.
$$\blanks(G;D_1,\dotsc,D_s|P)\leq \abs{X} + p' + \floor{n'/2} - 2.$$
Combining the above with inequality \eqref{eq:general-induction-target} and $\max(\abs{\cD}-\abs{\cD'},0)\geq 0$, we again find that inequality \eqref{eq:conditioned-induction-target} holds.

As we will prove that \eqref{eq:strange-induction-condition} and \eqref{eq:general-induction-target} hold for an exhaustive list of possible sequences $P$ of moves, we can deduce \Cref{thm:comparison-better} from \eqref{eq:conditioned-induction-target}.

\subsection{Maker plays a blank}
\label{subsec:maker-plays-blank}

Recalling that we say Maker's first move is at $u$, and Breaker's reply is at $v$, we will set $G'=G-\set{u,v}$.
We may assume that $n\geq 4$ (as otherwise we could conclude by \Cref{lem:base-case}), and so we see that $n'\geq 2$.
Note that as $\abs{\cD}\geq 1$, either $\abs{\cD'}\geq 1$, or, assuming that inequality \eqref{eq:general-induction-target} holds, $p-p'+(n-n')/2\geq \abs{X}+1$.
In either case, condition \eqref{eq:strange-induction-condition} holds.
Thus it suffices to prove inequality \eqref{eq:general-induction-target} holds.

Indeed, if Breaker plays a blank, then we need to prove the following:
\begin{align}
\label{eq:induction-blank-blank}
    p-p'+1 \geq \max(\abs{\cD} - \abs{\cD'},0).
\end{align}

Breaker, having played one blank, can remove one class marked for blanks.
It is possible that some $r\leq 2$ other classes marked for blanks have been completely coloured or played by blanks, and so will not be included in $\cD'$.
However, recalling that each $D_i$ is one of the classes $C_1,C_2,\dotsc,C_p$, we see that $p-p'\geq r$ and $\abs{\cD} - \abs{\cD'} \leq r+1$, from which we deduce inequality \eqref{eq:induction-blank-blank}.

If Breaker plays a colour, then we need to prove the following.
\begin{align}
\label{eq:induction-blank-col}
    p-p' \geq \max(\abs{\cD} - \abs{\cD'},0).
\end{align}

Here, the only classes marked for blanks which are removed are those which are completely coloured or played by blanks.
So for the same reason as the previous case we deduce inequality \eqref{eq:induction-blank-col}.
This completes this case of the induction.

\subsection{Maker plays a colour}
\label{subsec:maker-plays-colour}

We assume by relabelling that Maker played colour 1 at $u$, and again call the vertex played by Breaker $v$, and w.l.o.g. assume that either $v\in C_1$ or $v\in C_2$.
By relabelling colours, we may further assume that the play is of colour 1 or 2, or a blank.

We consider Maker's first $t$ turns, and so we remove $2t$ vertices from $G$ to form $G'$.
Recalling that $\abs{\cD}=0$ (as otherwise Maker would have played a blank), we find by simplifying inequality \eqref{eq:general-induction-target} and condition \eqref{eq:strange-induction-condition} that we need to prove one of the following inequalities.
\begin{align}
    p-p'+t &\geq \abs{X} & &\text{ if }(G'\not\simeq T(2r,r)\text{ and }n'\geq 6)\text{ or }\abs{\cD'}\geq 1\label{eq:induction-col},\\
    p-p'+t &\geq \abs{X}+1 & &\text{ otherwise.}\label{eq:induction-col-safe}
\end{align}

Note first that $p-p'+t \geq \abs{X}$ is implied by $p-p'+t \geq \abs{X} + 1$, and so proving the latter is sufficient, whether the conditions of \eqref{eq:induction-col} hold or not.
Indeed, when the unconditioned inequality of \eqref{eq:induction-col-safe} holds, either case \ref{case:matsumoto-base} of \Cref{lem:base-case}, or \Cref{thm:comparison-better} (by induction), and the observation that $\chi(G')\leq p'$ can be applied to $G'$ to immediately deduce the desired result.

We now present three simple observations, which we shall often refer back to.

\begin{observation}
\label{obs}
    If $\abs{\cD}=0$, then we may assume that all of the following hold, as otherwise we could deduce either inequality \eqref{eq:comparison-target} or inequality \eqref{eq:induction-col}.
    \begin{enumerate}
        \item There are at least two classes $C_i$ in $G$, i.e. $p\geq 2$. \label{obs:p-at-least-2}
        \item If $t=1$, then $n'\geq 6$. \label{obs:t=1}
        \item There is no class of size 1, i.e. for each $i\leq p$, $\abs{C_i}\geq 2$. \label{obs:no-size-1-class}
    \end{enumerate}
\end{observation}
\begin{proof}
    We deal with the three points in order.
    \begin{enumerate}
        \item If $p=1$, then $G$ has no edges and inequality \eqref{eq:comparison-target} follows immediately.
        \item If $t=1$ and $n'\leq 5$, then $n\in\set{6,7}$, and the result follows from \Cref{lem:base-case} \ref{case:matsumoto-6} or \ref{case:matsumoto-7}.
        \item If there was a class $\set{u}$ of size 1, then Maker will play at $u$. 
        If Breaker plays at $v$, then set $G'=G-\set{u,v}$, so $\abs{X}\leq 2$ and $p-p'\geq 1$.
        By point \ref{obs:t=1}, we know $n'\geq 6$.
        If $G'\simeq T(2r,r)$, then either Breaker replied in a class of size 1, so $p-p'=2$ and \eqref{eq:induction-col-safe} holds, or Breaker repeated colour 1, so $\abs{X}=1$ and \eqref{eq:induction-col-safe} holds, or we may mark some class for blanks, whence $\abs{\cD'}\geq 1$ and \eqref{eq:induction-col} holds.
    \end{enumerate}
\end{proof}

We now distinguish an exhaustive list of five cases depending on what Breaker's reply to Maker might be, dealing with each case separately.

\subsubsection{Breaker plays colour 2 and no class is completely coloured}
We see that Breaker plays colour 2 in either $C_1$ with $\abs{C_1} \geq 3$, or in $C_2$ with $\abs{C_2}\geq 2$.
Note that the only other situation in which Breaker can play colour 2 that we need to consider is when he plays in $C_1$ and $\abs{C_1}=2$, which is dealt with in \Cref{subsubsec:c1-of-size-2}.
In the case considered here, we need to track some further moves of the game, for which Maker's strategy is as follows.
For as long as Breaker plays previously-unplayed colours which can be copied, and no class is completely coloured, Maker keeps copying Breaker.

We now show that, while Maker is able to copy Breaker, no colour has been repeated, and no class has been filled, Breaker can always make a legal move using one of the remaining colours (as otherwise he would win).
Indeed, after Maker plays in round $t$ (so Maker has played $t$ times and Breaker $t-1$), the above assumptions imply that precisely $t$ different colours have been played.
Moreover, Breaker has always played previously-unplayed colours, and Maker has always copied Breaker, and so every played colour appears in precisely one class.
Thus, if some class has been played in, but still contains unplayed vertices, then Breaker can legally play in that class.
The only other case to consider is when all classes that have been played in are completely coloured.
However, this implies that only one class, $C_1$, has been played in, and it was completely coloured after Maker's $t$\ts{th} turn.
But then $\abs{C_1}$ is odd, and so, recalling Maker's strategy discussed at the start of \Cref{subsec:maker-strat}, we see that all classes have odd sizes at least $\abs{C_1}$, and so there must be more colours which have not yet been used, and so Breaker has a legal move.

Letting $U$ be the set of vertices played in the first $t$ rounds of the game, and $G'=G-U$, we take $t$ minimal so that either $\abs{X}=t$ or $p'<p$.

We consider five cases, depending on what vertices remain unplayed after the first $t$ rounds.
Note that, as we are removing many vertices at once, care must be taken to show that we have $n'\geq 2$; this is guaranteed by the first two of the following cases, which deal with $n'=0$ and $n'=1$ separately.
\begin{enumerate}
    \setlength\itemsep{1em}
    \item $G$ is completely coloured.
    We have $\abs{X}\leq t+1$, $n=2t$, and $p\geq 2$.
    However, we know by \Cref{obs} \ref{obs:no-size-1-class} that no class has size 1.
    Thus Breaker's last move must be into a class $C_i$ which already has some vertices coloured.
    As Breaker had not repeated a colour before his last move, if there are only $t$ colours, then no new colour is available.
    However, we know that the other colours in $C_i$ appear in no other class, so Breaker can be forced to repeat one such colour.

    \item $G$ has only one vertex left uncoloured, in some class $C_i$.
    In this case we do not use the induction hypothesis, but instead conclude directly.

    First note that $p\leq 3$, and $p\geq 2$ by \Cref{obs} \ref{obs:p-at-least-2}.
    If $p=3$, then on turn $t$ Maker and Breaker both filled a class in; say $C_1$ and $C_2$. 
    But then, as $\abs{C_3}\geq 2$, there must be some colour $c$ played only in $C_3$ before Breaker's final turn.
    If Breaker did not play colour $c$ on his final turn, Maker can give the final vertex colour $c$.
    Otherwise, colour $c$ will have been played two times, and so Maker can play a new colour at the final vertex.

    Now, as $p=2$, if we have $\abs{C_i}=2$ then we are done by \Cref{lem:base-case}.
    If not, then as at most one colour appears in two different classes (as Breaker's final play might have been to repeat a colour), there is some colour $c$ which has been played only in class $C_i$. 
    Then Maker may give the final vertex colour $c$, and we may directly see that inequality \eqref{eq:comparison-target} holds.
    
    Thus $t$ colours suffice and we directly see that \eqref{eq:comparison-target} holds.
    We may now assume in the remaining cases that $n'\geq 2$.
    
    \item Breaker repeats a colour before any class is filled.
    In this case $p=p'$, and $\abs{X}=t$.
    Then by assumption Breaker's first move was to play colour 2, so we know that $t\geq 2$, and thus some colour occurs in only one class, say $C_i$.
    The remainder of $C_i$ can then be marked for blanks, so $\abs{\cD'}\geq 1$, as required for condition \eqref{eq:induction-col} to hold.

    \item Some class $C_i$ is filled on Breaker's $t$\ts{th} move. 
    We have $p-p'\geq 1$ and $\abs{X}\leq t+1$, as Maker always copied Breaker's colour.
    Thus \eqref{eq:induction-col} holds, and we are done unless $p-p'=1$, $\abs{X}=t+1$, no class may be marked for blanks, and either $G'\simeq T(2r,r)$ or $n'\leq 5$, so assume that this is the case.

    As $\abs{X}=t+1$, we know that no colour appears in more than one class, so either some class may be marked for blanks, or all plays were in class $C_1$, and thus $\abs{C_1}$ is even, so $n$ is also even.
    By assumption, $\abs{C_1}\geq 3$, so $\abs{C_1}\geq 4$.
    Maker played in a class of minimum even size by assumption, and $p\geq 2$ by \Cref{obs} \ref{obs:p-at-least-2}, so as $n'\leq 5$, we must have $G'\simeq E_4$, the empty graph on four vertices.
    But in this case, inequality \eqref{eq:comparison-target} may be checked to hold, so we are again done.

    \item Some class $C_i$ is filled on Maker's $t$\ts{th} move, and Breaker then plays some colour, call it $c$, in a different class, w.l.o.g. $C_2$.
    We have that $\abs{X}\leq t+1$ and $p-p'\geq 1$, so $p-p'+t\geq\abs{X}$.
    Furthermore, we know that vertices in at least two classes have been played, so we must have either $p-p'=2$, or colour $c$ had been played prior to Breaker's $t$\ts{th} move (and so $\abs{X}\leq t)$, or colour $c$ appears only in $C_2$, so class $C_2$ may be marked for blanks (and $\abs{\cD'}\geq 1$), and so we satisfy \eqref{eq:induction-col-safe} in the first two cases, and \eqref{eq:induction-col} in the third.
\end{enumerate}

\subsubsection{Breaker plays a blank}

There are only two sub-cases here.

If Breaker plays a blank in $C_1$ and $\abs{C_1} = 2$, then set $G'=G-C_1$, so we have $p-p'=t=\abs{X}=1$, satisfying \eqref{eq:induction-col-safe} as required.
    
If Breaker plays a blank in any other case, then we set $G'=G-\set{u,v}$ and may mark $C_1$ for blanks, satisfying \eqref{eq:induction-col} as $\abs{\cD'}=1$.

\subsubsection{Breaker plays colour 1 in class \texorpdfstring{$C_2$}{C2}}

We may set $G'=G-\set{u,v}$, so $\abs{X}=1$, $n'\geq 6$ by \Cref{obs} \ref{obs:t=1}, and we satisfy \eqref{eq:induction-col} unless $G'\simeq T(2r,r)$.
But if $G'\simeq T(2r,r)$ then $\abs{C_1}=\abs{C_2}=3$, and by assumption Maker played at a vertex $u$ of maximal degree in $C_1$, which was still missing an edge to $C_2$.
Thus all other vertices in $C_1$ must be missing some edge when compared to the complete multipartite graph.

Assuming $G'\simeq T(2r,r)$, we see that all three vertices in $C_1$ had no edge to $v$.
Say $C_1=\set{u,x,y}$. Then Maker's strategy is to play colour 1 again, at vertex $x$.
Say that Breaker replies at $z\in C_i$, and set $G''=G-\set{u,v,x,z}$, noting that now $t=2$.
Let $p''$ be the number of the sets $C_1,\dotsc,C_p$ with nonempty intersection with $G''$.

If Breaker gives vertex $z$ colour 1, we have $\abs{X}=1$ when passing to $G''$.
Otherwise, Breaker w.l.o.g. gives $z$ colour 2. 
If $z$ was the last uncoloured vertex in $C_i$, then we have $\abs{X}=2$ and $p-p''\geq 1$.
If $z$ is not the last uncoloured vertex in $C_i$, then we may note that colour 2 has been played only at $z$ and so we may mark the class $C_i$ for blanks in $G''$.
Thus either \eqref{eq:induction-col-safe} holds or some class is marked for blanks, so $\abs{\cD'}\geq 1$ and \eqref{eq:induction-col} holds, regardless of whether $G''\simeq T(2(r-1),r-1)$ or not.

\subsubsection{Breaker plays in class \texorpdfstring{$C_1$ and $\abs{C_1}=2$}{C1 and |C1|=2}}
\label{subsubsec:c1-of-size-2}

Set $G'=G-C_1$.
Then we have removed at most 2 colours and exactly one class, and by \Cref{obs} \ref{obs:t=1}, $n'\geq 6$.
Note we cannot have $G'\simeq T(2r,r)$, as then $G$ would be a subgraph of $T(2(r+1),r+1)$ on $2(r+1)$ vertices, so \Cref{lem:base-case} part \ref{case:matsumoto-interesting} would apply to $G$ and we would be able to conclude our desired result directly.
Thus \eqref{eq:induction-col} holds, as required.
    
\subsubsection{Breaker plays colour 1 in class \texorpdfstring{$C_1$ and $\abs{C_1}\geq 3$}{C1 and |C1|>=3}}

Set $G'=G-\set{u,v}$, removing two vertices and one colour, and mark $C_1$ for blanks.
Then $t=\abs{X}=\abs{\cD'}=1$, thus satisfying \eqref{eq:induction-col}, as required.
\\

In every case we have verified that Maker's strategy as detailed in \Cref{subsec:maker-strat} leads to Maker winning the vertex colouring game with blanks.
Therefore our proof is complete and \Cref{thm:comparison-better} holds.

%----------------------------------------------------------%

\section{The marking game}
\label{sec:marking-game}

Matsumoto \cite{matsumoto2019difference} also asked whether an inequality similar to \Cref{thm:comparison-basic} holds for the so-called marking game, or graph colouring game, which is defined as follows.

\begin{definition}[The marking game]
    Some number $k$ is fixed.
    Maker and Breaker alternately \emph{mark} vertices of a graph $G=(V,E)$ by assigning them to a set $M$.
    A vertex may only be marked if it has at most $k-1$ marked neighbours.
    Maker wins if eventually $M=V$, whereas Breaker wins if some vertex becomes impossible to mark.
    The \emph{marking number} $m(G)$, or \emph{game colouring number} $\col_g(G)$, of a graph $G$ is the minimal $k$ for which Maker can win the marking game on $G$.
\end{definition}

In particular, Matsumoto asked whether $m(G)-\chi(G)\leq\floor{n/2}-1$.
We answer this question in the negative with the following result.

\begin{theorem}
    For any constant $\eps > 0$, there is some integer $n$ and a graph $G$ on $n$ vertices for which $m(G)-\chi(G)>(1-\eps)n$.
\end{theorem}

\begin{proof}
    Let integer $r$ be large enough that $2/r<\eps$.
    Then let $G=T(r^2,r)$ be the complete $r$-partite graph on $n=r^2$ vertices.
    Firstly, we know that $\chi(G)=r$.
    Next, consider the last vertex to be marked.
    It has $r(r-1)$ neighbours, all of which are already marked, and so $m(G)\geq r(r-1)$.

    Thus $m(G)-\chi(G)\geq r(r-2)=n(1-2/r)>(1-\eps)n$, as required.
\end{proof}

We may also note that the Tur\'{a}n graphs maximise the value $m(G)$ for fixed order and chromatic number, so the above result is in this sense sharp.

%----------------------------------------------------------%

\section{Conclusion and future work}
\label{sec:future}

In this paper, we have resolved the problem of Matsumoto \cite{matsumoto2019difference} to determine all graphs $G$ for which $\chrom(G)-\chi(G)=\floor{n/2}-1$.
As noted in \cite{matsumoto2019difference}, there are infinitely many graphs $G$ for which $\chrom(G)-\chi(G)=\floor{n/2}-2$; any complete bipartite graph $K_{r,r}$ minus a perfect matching suffices, as does the Tur\'{a}n graph $T(2r,r+1)$.
Because of this, we see that there is no stability around the equality cases of \Cref{thm:comparison-basic}.
There is however hope of some more general classification of near-equality cases, and to this end we pose the following questions.

\begin{question}
    Are there any infinite families of graphs for which $\chrom(G)-\chi(G)=\floor{n/2}-2$, besides the graphs $K_{r,r}$ minus a perfect matching and the Tur\'{a}n graphs $T(2r,r+1)$?
\end{question}

We also ask the following more general question.

\begin{question}
\label{question:minus-O-1}
    For which infinite families of graphs do we have $\chrom(G)-\chi(G)=\floor{n/2}-O(1)$?
\end{question}

One may try to answer \Cref{question:minus-O-1} by strengthening \Cref{thm:comparison-better}; the sharp version of this statement should include some $s$-dependence on the right-hand side.

We also in particular note that we have no families of graphs of odd order approaching these bounds.
We therefore also pose the following question.

\begin{question}
    What is the correct upper bound on $\chrom(G)-\chi(G)$ when we consider only graphs of (large) odd order?
\end{question}

In a similar vein, one could also consider the situation wherein Breaker starts (rather than Maker), and ask for bounds similar to the above for graphs of odd and even order.

Finally, we ask how the values of $\chrom(G)$ and $\blanks(G)$ can differ, which relates to a question of Zhu \cite{zhu1999game} from 1999.
Namely, if Maker wins the vertex colouring game with $k$ colours, must she also win with $k+1$?
(See \cite{hollom2024monotonicity} for some recent work of the author on this problem.)
Such a question can be resolved for the vertex colouring game with blanks by a simple imagination strategy, but such a proof does not work for the usual vertex colouring game.
Note that there are in fact graphs where being able to play a blank does in fact help Breaker, for example if $G$ is the six vertex graph formed by a four-cycle plus a pendant edge and an isolated vertex (for which $\blanks(G)=3$ and $\chrom(G)=2$).
However, it is unclear to the author what the answer to the following question might be.

\begin{question}
    Is there some function $f\from\NN\to\NN$ for which $f(k)> k$ and $\blanks(G)\leq f(\chrom(G))$ for all graphs $G$?
\end{question}

It seems that methods more detailed than those presented in this paper would be required to make progress on these problems.
Nevertheless, the author is hopeful that these questions are within reach of methods not too far beyond those shown here.

%----------------------------------------------------------%

\section{Declaration of Competing Interests}
The author declares that they have no known competing financial interests or personal relationships that could have appeared to influence the work reported in this paper.

%----------------------------------------------------------%

\section{Acknowledgements}
\label{sec:acknowledgement}

The author is supported by the Trinity Internal Graduate Studentship of Trinity College, Cambridge.
The author would like to thank both their supervisor, Professor B\'{e}la Bollob\'{a}s, and the anonymous referees, for their thorough readings of the manuscript and many valuable comments.

%----------------------------------------------------------%

\bibliographystyle{abbrvnat}  
\renewcommand{\bibname}{Bibliography}
\bibliography{main}

%----------------------------------------------------------%

\appendix
\section{Proof of Lemma \ref{lem:base-case}}
\label{sec:appendix}
\markboth{APPENDIX}{APPENDIX}

We now prove \Cref{lem:base-case} parts \ref{case:matsumoto-6}, \ref{case:matsumoto-7}, \ref{case:matsumoto-2}, and \ref{case:matsumoto-3}.
As these cases concern only graphs of order seven or less, they can be checked by exhaustive case analysis, but we provide proofs here for completeness.

Throughout our proofs, we partition of $G$ into independent sets, and consider the sequence of sizes of these sets, which we call classes.
We will label these classes as $C_1,C_2,\dotsc,C_p$ with $\abs{C_1}\geq\abs{C_2}\geq\dots\geq\abs{C_p}$, and define the \emph{class-size sequence} $\cC = \cC(G) = (\abs{C_1},\abs{C_2},\dotsc,\abs{C_p})$.

Throughout our proofs, we will often state that Maker's strategy is to \emph{copy Breaker}.
By this we mean that Maker should play in the same set $C_i$ as Breaker, and use the same colour. 
In the case that Breaker played a blank, Maker instead repeats any colour already used in $C_i$ if possible, or plays a new, previously-unused colour if not.

In both cases, we need to show that Maker wins with $p+1$ colours.
\Cref{lem:greedy} will be used many times; it tells us that if $G$ has at most $\ceil{\frac{p+1}{2}}$ vertices of degree at least $p+1$, then Maker wins.

\begin{proof}[Proof of \Cref{lem:base-case} part \ref{case:matsumoto-6}]
    Graphs with class-size sequence $(2,2,2)$ have already been dealt with in \Cref{lem:base-case} \ref{case:matsumoto-interesting}, so do not need to be considered here. 
    We consider all possible class-size sequences in turn, and show that $\blanks(G)\leq p+1$.
    
    \begin{itemize}
        \item $\cC=(5,1)$ or $\cC=(4,2)$. There are at most 2 vertices of degree at least 3, so \Cref{lem:greedy} tells us that Maker wins.
        
        \item $\cC=(3,3)$. If Maker does not win by \Cref{lem:greedy}, then some vertex $x$ must have degree 3. Breaker only wins if one class contains all three colours.
        Say $x\in C_1$.
        Maker plays colour 1 at $x$.
        Then however Breaker plays, Maker may copy Breaker.
        If this was in $C_1$, then we are done. If it was in $C_2$, then however Breaker next plays, Maker can play colour 1 again in $C_1$, so neither class can have all colours, and Maker wins.

        \item $\cC=(4,1,1)$. There are at most 2 vertices of degree at least 4, so Maker wins by \Cref{lem:greedy}.

        \item $\cC=(3,2,1)$. If Maker does not win by \Cref{lem:greedy}, then all vertices in classes $C_2$ and $C_3$ must have degree at least 4.
        Thus some $x\in C_1$ has degree 3.
        Maker plays 1 at $x$, and however Breaker replies, Maker may play colour 1 again in $C_1$.
        Then Maker may use her third move to colour the single vertex in $C_3$ if it is not already coloured, and win.

        \item $\cC=(3,1,1,1)$ or $\cC=(2,2,1,1)$ or $\cC=(2,1,1,1,1)$ or $\cC=(1,1,1,1,1,1)$. Maker wins by \Cref{lem:greedy}.
    \end{itemize}
    This completes the proof.
\end{proof}

\begin{proof}[Proof of \Cref{lem:base-case} part \ref{case:matsumoto-7}]
    Similarly to the previous case, we need to show that $\blanks(G)\leq p+1$.
    \begin{itemize}
        \item $\cC\in\set{(6,1),(5,2),(5,1,1),(4,1,1,1),(2,2,1,1,1),(2,1,1,1,1,1),\\(1,1,1,1,1,1,1)}$. \Cref{lem:greedy} tells us that Maker wins.
        \item $\cC=(4,3)$. 
        We will split into several further cases.
        First observe that if Maker can force $C_2$ to be filled with only two colours (i.e. some colour is repeated in $C_2$), then the third colour will remain playable in $C_1$, and so Maker will win.
        Note further that if all vertices in $C_1$ have degree at most 2, then so too does some vertex of $C_2$, and thus we are done by \Cref{lem:greedy}.
        We may thus assume that there is a vertex in $C_1$ of degree 3.

        If there is a $u\in C_2$ such that there is no $v\in C_1$ with $\Gamma(v) = C_2\setminus \set{u}$ (for example, a vertex of degree 4), then Maker plays colour 1 at $u$.
        However Breaker replies, Maker can play colour 1 again somewhere in $C_2$, and then on her next turn ensure that $C_2$ is fully coloured.
        Thus Maker again wins.

        Otherwise, assume that for every $u\in C_2$, there is a $v\in C_1$ with $\Gamma(v)=C_2\setminus\set{u}$.
        We may also assume by our previous observations that there is a vertex in $C_1$ of degree exactly 3.
        There is a unique graph satisfying this condition: $C_1$ consists of a vertex $w$ of degree 3 and three vertices $x,y,z$ of degree 2 with pairwise distinct neighbourhoods.
        In this case, Maker plays colour 1 at $w$; then colour 1 can always be played in $C_1$.
        If Breaker responds in $C_2$, w.l.o.g. playing colour 2, then Maker also plays colour 2 in $C_2$.
        Then colour 2 can always be played in $C_2$, so Maker wins.
        If Breaker instead responds in $C_1$, say at $x$ with colour 1 or 2, then Maker can play colour 3 at the vertex in the common neighbourhood of $y$ and $z$.
        Then colour 3 can always be played in $C_2$, so Maker will win.
        \item $\cC=(4,2,1)$. First note that any vertex in $C_1$ has degree at most 3, so can always be coloured when four colours are available.
        Maker first gives colour 1 to the unique vertex in $C_3$.
        On her second turn, she colours a vertex $u\in C_2$ arbitrarily.
        On her third turn, either $C_2$ is completely coloured, and we are done, or some vertex $v\in C_2$ is not coloured. 
        But $v$ is not adjacent to $u$, so at most three neighbours of $v$ are coloured, so $v$ can be coloured, and so all vertices in $C_1$ remain colourable.
        \item $\cC=(3,3,1)$. If Maker does not win by \Cref{lem:greedy}, then there must be a vertex $x$ of degree 4 in $C_1\union C_2$. Say it is in $C_1$.
        Then Maker gives $x$ colour 1, gives another vertex in $C_1$ colour 1 on her second turn, and then colours the one vertex in $C_3$ on her third turn if it is not already coloured.
        Then every other vertex remains colourable, and Maker wins.
        \item $\cC=(3,2,2)$. Maker plays colour 1 at an arbitrary vertex $u\in C_1$.
        Then for her next three moves, Maker plays in the same class as Breaker to fill it, either copying his colour if she can, or playing a new colour if she cannot.
        The only reason why Maker would be unable to copy Breaker's colour is if he played a colour that had already been played elsewhere.
        Each such pair of moves thus uses at most 1 new colour.
        Therefore Maker may win with 4 colours, as required.
        \item $\cC=(3,2,1,1)$. If Maker does not win by \Cref{lem:greedy}, then both vertices in $C_2$ must have edges to all of $C_1\union C_3\union C_4$.
        Maker begins by playing colour 1 at $u\in C_4$.
        If Breaker replies at $C_3$, then we are left in $K_{2,3}$ with at least three unplayed colours, and Maker wins.
        If Breaker replies with colour 1 at some vertex $v$, then apply \Cref{lem:subgraph-imagination}, and let $G'=G-\set{u,v}$.
        Note that as $u$ and $v$ are not adjacent and $v\notin C_3$, we have $v\in C_1$. 
        Thus $G'$ has at most one vertex of degree at least 4, and so by \Cref{lem:greedy}, Maker wins on $G'$ even without considering classes marked for blanks (which would only help her), as required.

        So we may assume that Breaker plays colour 2 in either $C_1$ or $C_2$. In either case, Maker copies Breaker.
        Then the only vertex which might become impossible to colour is the unique vertex in $v\in C_3$.
        Thus if Breaker colours $v$, we are done. If not, Maker plays at $v$, and we are done anyway.
        \item $\cC=(3,1,1,1,1)$. \Cref{lem:greedy} implies that the graph must be complete 5-partite.
        But then Maker may play colour 1 in class $C_1$ for both of her first two moves, and win.
    \end{itemize}
    This completes the proof.
\end{proof}

\begin{proof}[Proof of \Cref{lem:base-case} part \ref{case:matsumoto-2}]
    The only class-size sequences are $(2)$ and $(1,1)$.
    \begin{itemize}
        \item $\cC=(2)$, so $p=1$.
        If zero colours are available, then both players are forced to play blanks, which is allowed as $C_1$, the unique class, is marked for blanks.
        Thus we have $\blanks(G;V(G))=0=p-1$, as required.
        \item $\cC=(1,1)$, so $p=2$.
        Let $V(G)=\set{x,y}$.
        Then $\blanks(G,\set{x},\set{y})\leq\blanks(G,\set{x})=1=p-1$, as required; Maker plays a blank at $x$, and then Breaker is forced to play either colour 1 or a blank at $y$.
    \end{itemize}
    This completes the proof.
\end{proof}

\begin{proof}[Proof of \Cref{lem:base-case} part \ref{case:matsumoto-3}]
    There are three possible class-size sequences for us to consider: $(3)$, $(2,1)$, and $(1,1,1)$.
    \begin{itemize}
        \item $\cC=(3)$, so $p=1$.
        If zero colours are available, then both players are forced to play blanks, and so we have $\blanks(G;V(G))=0=p-1$, as required.
        \item $\cC=(2,1)$, so $p=2$.
        If $C_2$ is marked for blanks, then Maker plays a blank in it. Then colour 1 can be played at both remaining vertices.
        If the set of size 2 is marked for blanks, then Maker plays colour 1 in the set of size 1. Then both players are forced to play blanks in the set of size 2, so Maker wins.
        \item $\cC=(1,1,1)$, so $p=3$. 
        We need to show that $\blanks(G;D_1,\dotsc,D_s)\leq 2$ when $s\geq 1$.
        But if two colours are available, then Maker can play a blank on her first move, and then must win.
    \end{itemize}
    This completes the proof.
\end{proof}

\begin{proof}[Proof of \Cref{lem:base-case} part \ref{case:matsumoto-base}]
    Note that if Breaker never plays a blank, then the result follows from \Cref{thm:comparison-basic}, so we may assume that Breaker does in fact play a blank, and so at most $n-1$ colours are used.
    Noting further that the $\chi(G) = 1$ case is trivial, and that $n - 1\leq \floor{n/2} + 1$ for $n\in\set{2,3,4}$, only the case of $n=5$ remains.
    Moreover, if $n=5$ and $\chi(G)\geq 3$, then it would suffice to prove that $\blanks(G)\leq \floor{n/2} + 2$, which is again immediate.
    Thus it only remains to consider the case of $n=5$ and $\chi(G) = 2$, and prove that $\blanks(G)\leq 3$ here.

    Indeed, as Breaker must play a blank at some point, it suffices for Maker to force any colour to be repeated at any point.
    As $\chi(G) = 2$, there is some independent set $C_1\sseq G$ of size at least 3.
    Maker plays colour 1 in $C_1$.
    Then either Breaker plays colour 1, and we are immediately done, or Breaker plays a different colour or a blank.
    In the latter case, Maker can play colour 1 again in $C_1$, and so in every case a colour is repeated.
    This completes the proof.
\end{proof}

%----------------------------------------------------------%

\end{document}